\documentclass[letterpaper, 10 pt, conference]{ieeeconf}
\IEEEoverridecommandlockouts                              
\title{\LARGE\bf Extremum Seeking is Stable for Scalar Maps\\ that are Strictly  but Not  Strongly Convex}
\author{Patrick McNamee, Miroslav Krsti\'c, and Zahra Nili Ahmadabadi
\thanks{P. McNamee and Z. N. Ahmadabadi are with the  Department of Mechanical Engineering,
        San Diego State University, 
        San Diego, CA, USA
        {\tt\small pmcnamee5123@sdsu.edu} and 
        {\tt\small zniliahmadabadi@sdsu.edu}}
\thanks{M. Krsti\'c with the Department of Mechanical and Aerospace Engineering,
        University of California San Diego,
        La Jolla, CA, USA
        {\tt\small krstic@ucsd.edu}}
\thanks{Thus work was partially funded by Department of the Navy, Office of Naval Research under ONR award number N000142412269. Any opinions, findings, and conclusions or recommendations expressed in this material are those of the author(s) and do not necessarily reflect the views of the Office of Naval Research.}%
}

\usepackage{graphicx}
\usepackage{amsmath,amssymb}
\usepackage{mathptmx}
\usepackage{times}
\usepackage{cases}
\usepackage{cite}
\usepackage{tikz}
    \usetikzlibrary{shapes,arrows,calc,positioning,fit,automata}
\usepackage{xcolor}

\newcommand{\average}[1]{\overline{#1}}
\newcommand{\avgest}[1]{\average{#1}}
\newcommand{\avgerror}[1]{\error{\average{#1}}}
\newcommand{\comment}[1]{}

\newcommand{\costfunction}{J}
\newcommand{\error}[1]{\Tilde{#1}}
\newcommand{\estimate}[1]{\hat{#1}}
\newcommand{\gain}{k}
\newcommand{\gradient}{G}
\newcommand{\gradientfilter}{M}
\newcommand{\hessian}{H}
\newcommand{\hessianfilter}{N}
\newcommand{\invhessian}{{\Gamma}}
	\newcommand{\loginvhessian}{\gamma}

\newcommand{\limitvalue}[1]{#1^{*}}
\newcommand{\lyapunov}{V}
\newcommand{\optimal}[1]{{#1}^{*}}
\newcommand{\sensoroutput}{y}
\newcommand{\parameter}{\theta}
\newcommand{\realnumbers}{\mathbb{R}}
\newcommand{\set}[1]{\mathcal{#1}}
\DeclareMathOperator{\sgn}{sgn}

\newcommand{\stationary}[1]{#1_{*}}

\newcommand{\domain}{\set{D}}

\newtheorem{thm}{Theorem}
\newtheorem{lem}[thm]{Lemma}
\newtheorem{prop}[thm]{Proposition}
\newtheorem{asmp}{Assumption}
\newtheorem{defn}{Definition}

\tikzset{
  block/.style    = {draw, thick, rectangle, minimum height = 3em,
    minimum width = 3em},
  sum/.style      = {draw, circle},
  input/.style    = {coordinate},
  output/.style   = {coordinate},
  gain/.style = {
  	draw, 
    isosceles triangle,
    isosceles triangle apex angle=50,
    minimum height = 3.0em,
    outer sep=0},
}

\newcommand{\convolve}{\Huge$\otimes$}
\newcommand{\suma}{\Large$\bigoplus$}

\pgfdeclarelayer{background}
\pgfdeclarelayer{foreground}
\pgfsetlayers{background,main,foreground}

\begin{document}

\maketitle

\begin{abstract}
For a map that is strictly but not strongly convex, model-based gradient extremum seeking has an eigenvalue of zero at the extremum, i.e., it fails at exponential convergence. Interestingly, perturbation-based model-free extremum seeking has a negative Jacobian, in the average, meaning that its (practical) convergence is exponential, even though the map's Hessian is zero at the extremum. While these observations for the gradient algorithm are not trivial, we focus in this paper on an even more nontrivial study of the same phenomenon for Newton-based extremum seeking control (NESC). 

NESC is a second-order method which corrects for the unknown Hessian of the unknown map, not only in order to speed up parameter convergence, but also (1) to make the convergence rate user-assignable in spite of the unknown Hessian, and (2) to equalize the convergence rates in different directions for multivariable maps. Previous NESC work established stability only for maps whose Hessians are strictly positive definite everywhere, so the Hessian is invertible everywhere. For a scalar map, we establish the rather unexpected property that, even when the map behind is strictly convex but not strongly convex, i.e., when the Hessian may be zero, NESC guarantees practical asymptotic stability, semiglobally. While a model-based Newton-based algorithm would run into non-invertibility of the Hessian, the perturbation-based NESC, surprisingly, avoids this challenge by leveraging the fact that the average of the perturbation-based Hessian estimate is always positive, even though the actual Hessian may be zero.
\end{abstract}

\section{Introduction}

Extremum seeking control (ESC) is form of parameter optimization for continuous time systems first proposed in the 1920s and popularized starting in the 2000s \cite{ref:tan-2010}. ESC attempts to optimize some map $\costfunction$ of continuous parameters $\parameter$, either scalars or vectors, to some real value sensor output $\sensoroutput$. There are numerous applications for ESC such as acoustic source seeking, adaptive control of combustion instabilities \cite{ref:banaszuk-2000}, fly wheel control \cite{ref:walsh-2000}, and model-free stabilization \cite{ref:scheinker-2016}. One of the advantages of ESCs is that they are designed for maps which are {\itshape a priori} unknown at system deployment. This differs from other optimization based continuous time dynamical systems which either evolve by vector fields that are the step size limits of difference equations \cite{ref:su-2014} or assumed that the derivatives of the function are known such as \cite{ref:nesic-2012}. ESC accomplishes this by both locally exploring the parameter space to estimate local map information as well as simultaneously updating the estimate of the optimal parameters.
	
As ESC simultaneously explores the local parameter space and updates its optimal parameter estimate, proofs of ESC stability tend to rely on a representative system which has a simpler analysis and shadows the full system. Two common techniques for constructing these representative systems are direct Averaging theory \cite{ref:ghaffari-2012} and Lie Bracket Analysis \cite{ref:durr-2013,ref:labar-2019,ref:todorovski-2022}. Gradient-based extremum seeking control (GESC) use information on the gradient of the map to determine their parameter updates. The convergence rate for GESC around an extremum point $\optimal{\parameter}$ is dependent on the magnitude of eigenvalues of the Hessian \cite{ref:ghaffari-2012}. This convergence rate may diminish for small eigenvalues, leading to unacceptably slow convergence rate. This rate dependence on the unknown Hessian is the driving motivation behind the NESC algorithm. The NESC corrects the convergence rate by multiplying the gradient estimate vector with an estimated inverse Hessian matrix to achieve a user assigned, local convergence rate that is completely independent of the unknown Hessian. The first proposed NESC is a scalar version in \cite{ref:moase-2009} while the first multivariate case is given in \cite{ref:ghaffari-2012}. Both of these works assume that the Hessian of the map was positive definite in the minimization problem although this assumption is frequent in previous works \cite{ref:yin-2018, ref:labar-2019}. The reason behind this assumption was that stability proofs rely on average system estimates converging to their true values as the local exploration amplitude diminishes and the exploration rate increases. Thus a positive definite Hessian generates a locally quadratic map that is locally strongly convex with local practical exponential stability to the extremum point. 

Reference \cite{ref:labar-2019} was the first to give a NESC that is sGPUAS on maps with invertible Hessians, but the local stability of this NESC around the global minimum was not user assignable, i.e., the local stability depends on the eigenvalues of the unknown Hessian rather than solely parameters chosen by a system designer. An improved NESC formulation in \cite{ref:todorovski-2023} corrects this, allowing the NESC to still be sGPUAS while allowing the local stability to be fully assignable. Both \cite{ref:labar-2019} and \cite{ref:todorovski-2023} make the assumption that the Hessian of the map was always invertible so it may be surprising that the perturbation-based scalar NESC avoids this challenge. The challenge is overcome by leveraging the fact that the average of the perturbation-based Hessian estimate is always positive as given in Prop.~\ref{prop:average-hessian-strictly-positive}, even though the actual Hessian may be zero.
	
{\itshape Main Contribution:} The main contribution of this work is to show the relaxation of the strict requirement $\costfunction''~>~0$ when considering NESC practical stability for static scalar maps. For scalar maps which are strictly convex but not necessarily strongly convex, the NESC algorithm is shown to be sGPUAS whenever there exists a global minimizer.

\section{Perturbations Remove Strong Convexity Restriction in Gradient ESC}

\begin{figure}[t]
    \centering
    \resizebox{\columnwidth}{!}{
        \begin{tikzpicture}
    \draw
        node[block] at (0,0) (costfunction) {$\costfunction(\parameter)$}
        node[below right=of costfunction] (gradientfilter) {\convolve}
        node[right=0.75cm of gradientfilter] (M) {$M(t)$}
        node[gain, left=3cm of gradientfilter, rotate=180] (gain) {\rotatebox{-180}{$\frac{\gain}{s}$}}
        node[left=2cm of gain] (perturbation) {\suma}
        node[below=0.5cm of perturbation] (S) {$S(t)$}
        ;

    \draw[-latex] (costfunction.east) --++(3,0) node[anchor=west] {$\sensoroutput$};
    \draw[-latex, shorten >=-0.15cm] (costfunction.east) -| (gradientfilter.north);
    \draw[-latex, shorten >=-0.19cm] (M.west) -- (gradientfilter.east);
    \draw[-latex, shorten >=-0.16cm] (gain.east) -- node[anchor=south] {$\estimate{\parameter}$} (perturbation.east);
    \draw[-latex, shorten <=-0.15cm] (perturbation.north) |- node[pos=0.75, anchor=south] {$\parameter$} (costfunction.west);
    \draw[-latex, shorten <=-0.19cm] (gradientfilter.west) -- node[anchor=south] {$\estimate{\gradient}$}(gain.west);
    \draw[-latex, shorten >=-0.15cm] (S.north) -- (perturbation.south);
\end{tikzpicture}
    }%
    \caption{Scalar GESC Algorithm}
    \label{fig:gesc-diagram}
\end{figure}
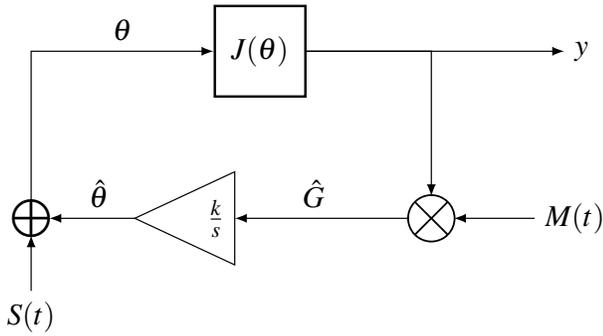

Consider the model-based GESC
\begin{equation}
    \label{eq:gesc-model-based}
    \dot{\estimate{\parameter}} = -\gain J'(\estimate{\parameter})
\end{equation}
seeking to minimize $\costfunction: \realnumbers\to\realnumbers$ where $\estimate{\parameter}$ is the estimate of the optimal parameter $\optimal{\parameter}$ and $\gain>0$ is a gain. The Jacobian of \eqref{eq:gesc-model-based} around an equilibrium $\stationary{\parameter}$ is $-\gain J''(\optimal{\parameter})$ which may be zero for strictly convex maps. 

A model-free GESC algorithm updates its parameter estimate using the first order differential equation
\begin{equation}
    \label{eq:gesc-parameter-update-ode}
    \dot{\estimate{\parameter}} = -\gain \estimate{\gradient}
\end{equation}
where $\estimate{\gradient}$ is the perturbation-based estimate of the gradient $\costfunction'$. This gradient estimate is $\estimate{\gradient} = \gradientfilter(t) \sensoroutput$ where $\sensoroutput$ is a perturbed sensor output and $\gradientfilter(t)$ a time varying filter. The perturbed sensor output is $\sensoroutput = \costfunction(\estimate{\parameter}(t)+S(t))$ were $S(t) = a\sin(\omega t)$ with $a\neq 0$ and $\omega>0$ while $M(t) = 2\sin(\omega t)/a$. The average system associated with \eqref{eq:gesc-parameter-update-ode} is
\begin{equation}
    \label{eq:gesc-average-parameter-update-ode}
    \dot{\average{\parameter}} = -\gain \average{\gradient}(\average{\parameter})
\end{equation}
where $\average{\gradient}$ is the averaged gradient estimate defined as
\begin{equation}
	\label{eq:average-gradient-estimate}
	\avgest{\gradient}(\avgest{\parameter}) = \frac{1}{a\pi}\int_{0}^{2\pi} \sin(\tau)\costfunction(\avgest{\parameter} + a\sin\tau)d\tau.
\end{equation}
The Jacobian of \eqref{eq:gesc-average-parameter-update-ode} around an equilibrium point $\stationary{\parameter}$ is $-\gain\average{\gradient}'(\stationary{\parameter})$, which in contrast to the model-based GESC is always strictly positive for strictly convex maps by Prop.~\ref{prop:average-gradient-estimate-strictly-monotonically-increasing}. The model-free perturbation-based GESC retains local exponential stability around equilibrium points on strictly convex maps.

\section{Newton ESC and Its Average}

\begin{figure}[tb]
    \centering
    \resizebox{\columnwidth}{!}{
        \begin{tikzpicture}
    \draw
        node[block] at (0,0) (costfunction) {$\costfunction(\parameter)$}
        node[block, below=2cm of costfunction, left of=costfunction] (directionestimate) {$-\estimate{\invhessian}\estimate{\gradient}$}
        node[right=2cm of directionestimate] (gradientfilter) {\convolve}
        node[right=0.75cm of gradientfilter] (M) {$M(t)$}
        node[gain, left=3cm of directionestimate, rotate=180] (gain) {\rotatebox{-180}{$\frac{\gain}{s}$}}
        node[left=2cm of gain] (perturbation) {\suma}
        node[below=0.5cm of perturbation] (S) {$S(t)$}
        node[block, below=1cm of directionestimate, xshift=3cm] (inversehessianode) {$\frac{d}{d t}\estimate{\invhessian} = \omega_l \estimate{\invhessian}\left(1 -  \estimate{\invhessian}\estimate{\hessian}\right)$}
        node[right=1cm of inversehessianode] (hessianfilter) {\convolve}
        node[right=0.75cm of hessianfilter] (N) {$N(t)$};

    \draw[-latex] (costfunction.east) --++(5,0) node[anchor=west] {$\sensoroutput$};
    \draw[-latex, shorten >=-0.15cm] (costfunction.east) -| (gradientfilter.north);
    \draw[-latex, shorten >=-0.15cm] (costfunction.east) -| (hessianfilter.north);
    \draw[-latex, shorten >=-0.19cm] (M.west) -- (gradientfilter.east);
    \draw[-latex, shorten >=-0.19cm] (N.west) -- (hessianfilter.east);
    \draw[-latex, shorten >=-0.15cm] (S.north) -- (perturbation.south);
    \draw[-latex, shorten <=-0.19cm] (hessianfilter.west) -- node[anchor=south] {$\estimate{\hessian}$} (inversehessianode.east);
    \draw[-latex] (inversehessianode.west) -| node[anchor=south, pos=0.3] {$\estimate{\invhessian}$} (directionestimate.south);
    \draw[-latex] (directionestimate.west) -- (gain.west);
    \draw[-latex, shorten >=-0.16cm] (gain.east) -- node[anchor=south] {$\estimate{\parameter}$} (perturbation.east);
    \draw[-latex, shorten <=-0.15cm] (perturbation.north) |- node[pos=0.75, anchor=south] {$\parameter$} (costfunction.west);
    \draw[-latex, shorten <=-0.19cm] (gradientfilter.west) -- node[anchor=south] {$\estimate{\gradient}$}(directionestimate.east);
\end{tikzpicture}
    }%
    \caption{Scalar NESC Algorithm}
    \label{fig:nesc-diagram}
\end{figure}
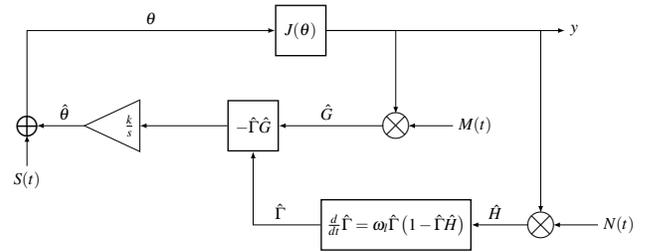

An improvement to gradient-based methods is to account for the second derivative of the map to compensate for map curvature. The system differential equations for a scalar NESC \cite{ref:ghaffari-2012} are
\begin{eqnarray}
    \label{eq:parameter-estimate-ode}
    \dot{\estimate{\parameter}} &=& -\gain\estimate{\invhessian}\estimate{\gradient} \\
    \label{eq:invhessian-estimate-ode}
    \dot{\estimate{\invhessian}} &=& \omega_l \estimate{\invhessian}\left(1 - \estimate{\invhessian}\estimate{\hessian}\right)
\end{eqnarray}
where $\omega_l > 0$ is a filter gain, $\estimate{\hessian}$ is the estimate of the Hessian, and $\estimate{\invhessian}$ is the Ricatti based low pass filter formulation of the inverse of $\estimate{\hessian}$. The domain for this system is the half-plane $\domain=\left\lbrace(\estimate{\parameter},\ \estimate{\invhessian})\ \vert\ \estimate{\invhessian} > 0\ \right\rbrace$ since the unstable equilibrium at $\estimate{\invhessian} = 0$ in Eq.~\eqref{eq:invhessian-estimate-ode} prevents $\estimate{\invhessian}$ from changing sign \cite{ref:ghaffari-2012}. Another time varying signal $\hessianfilter(t) = \frac{16}{a^2}\left(\sin^2(\omega t) - \frac{1}{2}\right)$ is used to estimate the Hessian $\estimate{\hessian}=\hessianfilter\sensoroutput$ from the sensor output $\sensoroutput$. The average system for the model-free NESC is
\begin{eqnarray}
    \label{eq:average-system-parameter-update-ode}
	\dot{\average{\parameter}} &=& -\gain \average{\invhessian}\cdot\average{\gradient}(\average{\parameter})\\
	\label{eq:average-system-invhessian-ode}
	\dot{\average{\invhessian}} &=& \omega_l \average{\invhessian}\left(1 -\average{\invhessian}\cdot\avgest{\hessian}(\average{\parameter})\right)
\end{eqnarray}
where $\average{\hessian}$ is the averaged Hessian estimate defined as
\begin{equation}
    \label{eq:average-hessian-estimate}
    \avgest{\hessian}(\estimate{\parameter}) = \frac{8}{a^2\pi} \int_{0}^{2\pi}\left(\sin^2(\tau) - \frac{1}{2}\right)\costfunction(\estimate{\parameter} + a\sin\tau)d\tau.
\end{equation}
Equilibrium points of \eqref{eq:average-system-parameter-update-ode} are the same as \eqref{eq:gesc-average-parameter-update-ode}, but now there is the Hessian and averaged Hessian estimate to consider. For strictly convex functions such as quartic functions, the Hessian can be zero so \eqref{eq:average-system-parameter-update-ode} with model-based information $\costfunction''$ is exponential unstable in $\estimate{\invhessian}$ whenever the gradient and Hessian of the map is zero. In stark contrast, the average the perturbation-based Hessian estimate is strictly positive for strictly convex functions by Prop.~\ref{prop:average-hessian-strictly-positive}, preventing this instability.

\section{Problem Formulation and Assumptions}
	\label{sec:problem-formulation-and-assumptions}

We considers the maps under the following three assumptions.
 
\begin{asmp}
    \label{asmp:twice-differentiable}
    The map $\costfunction$ is a twice differentiable function.
\end{asmp}
\begin{asmp}
    \label{asmp:strictly-convex}
    The map $\costfunction$ is a strictly convex function, that is for any distinct pairs of $\parameter_1,\ \parameter_2 \in \realnumbers$ and $\forall \lambda \in (0,1)$ the following inequality holds over the convex line segment \cite{ref:boyd-2004}.
	\begin{equation}
		\label{eq:strict-convexity}
		\costfunction(\lambda \parameter_1 + (1-\lambda)\parameter_2) < \lambda \costfunction(\parameter_1) + (1-\lambda)\costfunction(\parameter_2)
	\end{equation}
\end{asmp}
\begin{asmp}
    \label{asmp:global-minimizer}
    The map $\costfunction$ has a global minimizer $\optimal{\parameter}$ where:
    \begin{enumerate}
        \item $\costfunction(\optimal{\parameter}) < \costfunction(\parameter)$ for all $\parameter\neq\optimal{\parameter}$;
        \item $\costfunction'(\optimal{\parameter}) = 0$;
        \item and $\costfunction'(\parameter) \neq 0$ for all $\parameter\neq\optimal{\parameter}$.
    \end{enumerate}
\end{asmp}

Under the stated assumptions about the map, this work proves that the NESC is sGPUAS where the definitions of practical stability are given in Appendix \ref{sec:stability-definitions} for reference.

\begin{thm}[The NESC system is sGPUAS]
	\label{thm:nesc-sgpuas}
	Let the system described in equations \eqref{eq:parameter-estimate-ode}-\eqref{eq:invhessian-estimate-ode} be acting on a static, scalar map $\costfunction$ satisfying Assumptions \ref{asmp:twice-differentiable}-\ref{asmp:global-minimizer}. Then in the error variables
	\begin{equation}
		\label{eq:change-of-variables}
		(\error{\parameter},\ \error{\loginvhessian}) = \left(\estimate{\parameter} - \stationary{\parameter},\ \ln\left(\estimate{\invhessian}\avgest{\hessian}\right)\right)
	\end{equation}
	where $\stationary{\parameter}\in\realnumbers$ is a unique point, the NESC system is sGPUAS to the origin. Furthermore, the point $(\hat\theta,\hat\Gamma)=(\stationary{\parameter}, 1/\hessian(\stationary{\parameter}))$  is practically exponentially stable.
\end{thm}

\section{Motivating Example}

To demonstrate the complexity that can arise with the NESC on strictly convex functions, consider the cost function
\begin{equation}
	\label{eq:example-cost-function}
	\costfunction(\parameter) = \parameter^2 \left(\exp(\parameter) - 1\right)^2
\end{equation}
which satisfies Assumptions A\ref{asmp:twice-differentiable}-A\ref{asmp:global-minimizer} and has a global minimum at $\optimal{\parameter} = 0$. This cost function is strictly convex by Prop.~\ref{prop:strictly-positive-hessian-integral} in Appendix \ref{sec:convexity-lemmas} as the Hessian is zero only at the minimum but positive everywhere else. Qualitatively, the example function behaves as $\costfunction(\parameter)\sim \parameter^2$ as $\parameter\to -\infty$ but behaves as $\costfunction(\parameter)\sim \parameter^2\exp(2\parameter)$ as $\parameter\to\infty$. From the example trajectories in Fig.~\ref{fig:numerical-example}, it becomes clear that there is a point $(\stationary{\parameter},\stationary{\invhessian})$ which is practically stable but $\stationary{\parameter} \neq \optimal{\parameter}$. A further examination of selected trajectories in Fig.~\ref{fig:numerical-example-time-history} illustrates the local practical stability near $\stationary{\parameter}$.

\begin{figure}
	\includegraphics[width=3.25in]{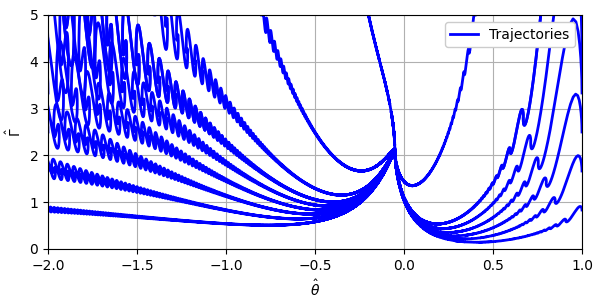}
	\caption{Numerical simulated trajectories of the NESC using Eq.~\eqref{eq:example-cost-function}. Simulation parameters were $a=0.5$, $\omega = 10$, and $\gain=\omega_l= 0.001$.}
	\label{fig:numerical-example}
\end{figure}
\begin{figure}
    \centering
    \includegraphics[width=3.25in]{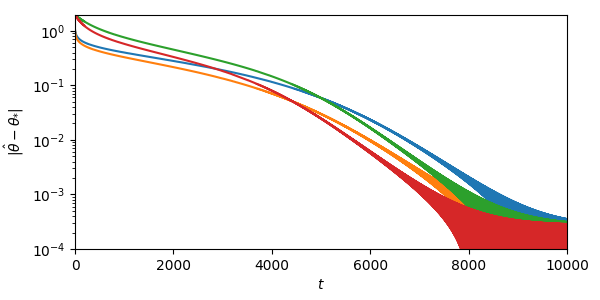}
    \caption{Numerical simulated parameter estimates of the NESC using Eq.~\eqref{eq:example-cost-function}. Simulation parameters were $a=0.5$, $\omega = 10$, and $\gain=\omega_l= 0.001$. Initial conditions of $(\estimate{\parameter},\estimate{\invhessian})$ were (1,5/6), (1,5/3), (1,5/6) and (1,5/3) for the blue, orange, red and green lines respectively. After 5,000 seconds, all the parameter estimates increase their convergence rate and appear to converge exponentially to a neighborhood of $\stationary{\parameter}$, with $\Vert \estimate{\parameter}(t) - \stationary{\parameter}\Vert <$10e-3 for $t > 8,000$ seconds.}
    \label{fig:numerical-example-time-history}
\end{figure}

\section{Main Results}
	\label{sec:main-results}

\subsection{Averaged Derivative Estimates}

To understand the stability of the NESC, one must understand the behavior of the average systems with the averaged derivative estimates $\avgest{\gradient}$ and $\avgest{\hessian}$ of the scalar map. The average system uses the average states $(\average{\parameter},\average{\invhessian})$, the averaged $\avgest{\hessian}$ defined earlier in \eqref{eq:average-hessian-estimate}, the average gradient estimate $\avgest{\gradient}$ defined earlier in \eqref{eq:average-gradient-estimate}. The average system dynamics in \eqref{eq:average-system-parameter-update-ode} and \eqref{eq:average-system-invhessian-ode} have a stable equilibrium when $\avgest{\gradient}(\avgest{\parameter}) = 0$ and $\avgest{\invhessian} = \left[\avgest{\hessian}(\avgest{\parameter})\right]^{-1}$ \cite{ref:ghaffari-2012}, so how many points $\stationary{\parameter}$ exists where $\avgest{\gradient}(\stationary{\parameter}) = 0$? First we show that $\avgest{\gradient}$ is a monotonically increasing function with respect to $\average{\parameter}$ for strictly convex scalar maps. In this case, $\stationary{\parameter}$ would be a unique point if it exists.

\begin{prop}
\label{prop:average-gradient-estimate-strictly-monotonically-increasing}
Under Assumptions \ref{asmp:twice-differentiable} and \ref{asmp:strictly-convex}, the average system gradient estimate $\avgest{\gradient}$ defined in\eqref{eq:average-gradient-estimate} is a strictly monotonically increasing function with respect to $\avgest{\parameter}$.
\end{prop}
\begin{proof}
Rearranging \eqref{eq:average-gradient-estimate} gives in a more suitable form for differentiation
\begin{equation}
	\avgest{\gradient}(\avgest{\parameter}) = \int_{0}^{\pi}\frac{\sin(\tau)}{a\pi}\left(\costfunction(\avgest{\parameter} + a\sin\tau) - \costfunction(\avgest{\parameter} - a\sin\tau) \right)d\tau
\end{equation}
which when differentiating and then replacing the difference with an integration yields
\begin{equation}
	\avgest{\gradient}'(\avgest{\parameter}) = \frac{1}{a\pi}\int_0^{\pi}\sin(\tau)\int_{\avgest{\parameter}-a\sin(\tau)}^{\avgest{\parameter}+a\sin(\tau)} \costfunction''(\phi)d\phi d\tau > 0.
\end{equation} The strict inequality is a consequence of Prop.~\ref{prop:strictly-positive-hessian-integral} and $\sin\tau > 0$ for all $\tau\in(0,\ \pi)$. With the derivative of $\avgest{\gradient}$ being strictly positive, $\avgest{\gradient}$ must be a strictly monotonically increasing function with respect to $\avgest{\parameter}$.
\end{proof}

In addition to showing the monotonicity of $\avgest{\gradient}$, we establish the existence and uniqueness of $\stationary{\parameter}$ when the scalar map has a global minimizer.

\begin{prop}
\label{prop:average-gradient-estimate-unique-zero}
Under Assumptions \ref{asmp:twice-differentiable}-\ref{asmp:global-minimizer} the gradient estimate in the average system $\avgest{\gradient}$ as defined by \eqref{eq:average-gradient-estimate} has a unique zero $\stationary{\parameter}\in(\optimal{\parameter} - a,\optimal{\parameter} + a)$.
\end{prop}
\begin{proof}
The proof follows by providing an upper and lower bound to $\avgest{\gradient}$ and apply Assumption~\ref{asmp:global-minimizer} to the found bounds. First, \eqref{eq:average-gradient-estimate} is simplified by periodicity to
\begin{equation}
	\label{eq:average-gradient-estimate-alt}
	\avgest{\gradient} = \frac{2}{a\pi}\int_{-\pi/2}^{\pi/2} \sin(\tau)\costfunction(\avgest{\parameter} + a\sin(\tau))d\tau
\end{equation}
Next, we introduce an auxiliary function $f$ using a convex line segment and support line
\begin{subnumcases}{f(\phi)=}
	\lambda\costfunction\left(\average{\parameter} - a\right) + \left(1 - \lambda\right)\costfunction(\average{\parameter}) & for $\phi < \avgest{\parameter}$ \\
	\costfunction(\average{\parameter}) + \costfunction'(\average{\parameter})\cdot(\parameter - \avgest{\parameter}) & for $\phi \geq \avgest{\parameter}$
\end{subnumcases}
where $\lambda = (\average{\parameter} - \phi)/a$ to form a lower, strict bound on the integral in \eqref{eq:average-gradient-estimate-alt} by replacing $\costfunction$ with $f$ in the integrand. The lower bounds are justified by the convex line segment and support line regions. The inequality
	\begin{equation}
		\sin(\tau)f(\avgest{\parameter} + a\sin(\tau)) \leq \sin(\tau)\costfunction(\avgest{\parameter} + a\sin(\tau))
	\end{equation}
may then be formed with equality met only when $\sin\tau = 0$. Hence, replacing $f$ for $\costfunction$ in \eqref{eq:average-gradient-estimate-alt} forms the lower bounds of $\avgest{G}$
\begin{eqnarray}
	\avgest{\gradient}(\avgest{\parameter}) &>& \frac{2}{a\pi}\int_{-\pi/2}^{\pi/2} \sin(\tau)f(\avgest{\parameter} + a\sin(\tau))d\tau \\
	&=& \frac{1}{2}\left(\costfunction'(\avgest{\parameter}) + \frac{1}{a}\int_{\avgest{\parameter}-a}^{\avgest{\parameter}} \costfunction'(\phi)d\phi\right) \\
	\label{eq:average-gradient-estimate-lower-bound}
	& > & \costfunction'(\avgest{\parameter} - a)
\end{eqnarray}
with the last strict inequality a result of Prop.~\ref{prop:strictly-positive-hessian-integral}. The upper bound for $\avgest{\gradient}$ follows a similar logic and argument to arrive at
\begin{eqnarray}
\avgest{\gradient}(\avgest{\parameter}) &<& \frac{1}{2}\left(\costfunction'(\avgest{\parameter}) + \frac{1}{a}\int_{\avgest{\parameter}-a}^{\avgest{\parameter}} \costfunction'(\phi)d\phi\right) \\
	\label{eq:average-gradient-estimate-upper-bound}
	& < & \costfunction'(\avgest{\parameter} + a)
\end{eqnarray}
With the upper and lower bounds established on $\avgest{\gradient}$, we combine both \eqref{eq:average-gradient-estimate-lower-bound} and \eqref{eq:average-gradient-estimate-upper-bound} along with Assumption~\ref{asmp:global-minimizer} to show that 
\begin{eqnarray}
\avgest{\gradient}(\optimal{\parameter} - a) < 0 < \avgest{\gradient}(\optimal{\parameter} + a)
\end{eqnarray}
By the Intermediate Value Theorem, there is a zero $\stationary{\parameter}\in(\optimal{\parameter}-a,\ \optimal{\parameter}+a)$ and this zero is unique as $\avgest{\gradient}$ is a strictly monotonically increasing function by Prop.~\ref{prop:average-gradient-estimate-strictly-monotonically-increasing}.
\end{proof}

The existence and uniqueness of $\stationary{\parameter}$ does establish the uniqueness of a stable equilibrium for the NESC system but it does not establish by itself the existence of such an equilibrium. If one considers \eqref{eq:average-system-invhessian-ode} with $\average{\hessian}(\stationary{\parameter}) = 0$, then $\average{\invhessian}\to\infty$ as $t\to\infty$. This divergence would prevent the existence of a stable equilibrium. However, the averaging process in \eqref{eq:average-hessian-estimate} ensures that $\avgest{\hessian}$ is strictly positive everywhere.

\begin{prop}
\label{prop:average-hessian-strictly-positive}
Under Assumptions \ref{asmp:twice-differentiable} and \ref{asmp:strictly-convex}, the Hessian estimate in the average system $\avgest{\hessian}$ as defined by \eqref{eq:average-hessian-estimate} is a strictly positive function for any $a\neq 0$.
\end{prop}
\begin{proof}
Similar to the proof of Prop.~\ref{prop:average-gradient-estimate-unique-zero}, \eqref{eq:average-hessian-estimate} is simplified by periodicity to
\begin{equation}
	\label{eq:average-hessian-estimate-alt}
	\avgest{\hessian}(\avgest{\parameter}) = \frac{8}{a^2\pi}\int_{-\pi/2}^{\pi/2} \left(-\cos(2\tau)\right)\costfunction(\avgest{\parameter} + a\sin\tau)d\tau
\end{equation}
and uses another auxiliary function $h$ based on parameterized convex line segments $\ell$ and parameterized support lines $s$.
\begin{eqnarray}
	\ell(\phi;\ \phi_{1}, \phi_{2}) &=& \left(\frac{\phi_2 - \phi}{\phi_2 - \phi_1}\right)\costfunction(\phi_1) \nonumber \\ &&+ \left(\frac{\phi - \phi_1}{\phi_2 - \phi_1}\right)\costfunction(\phi_2) \\
	s(\phi;\ \phi_s) &=& \costfunction(\phi_s) +\costfunction'(\phi_s)\cdot\left(\phi - \phi_s\right)
\end{eqnarray}
\begin{subnumcases}{h(\phi)=}
	s\left(\phi;\ \average{\parameter}-\frac{a}{\sqrt{2}}\right) & $\phi\in\left[\average{\parameter}-a,\average{\parameter}-\dfrac{a}{\sqrt{2}}\right]$ \\
	\ell\left(\phi;\ \avgest{\parameter} - \frac{a}{\sqrt{2}}, \average{\parameter}\right) & $\phi\in\left(\average{\parameter}-\dfrac{a}{\sqrt{2}},\average{\parameter}\right]$ \\
	\ell\left(\phi;\ \average{\parameter}, \avgest{\parameter} + \frac{a}{\sqrt{2}}\right) & $\phi\in\left(\average{\parameter},\average{\parameter}+\dfrac{a}{\sqrt{2}}\right]$ \\
	s\left(\phi;\ \average{\parameter}+\frac{a}{\sqrt{2}}\right) & $\phi\in\left(\average{\parameter}+\dfrac{a}{\sqrt{2}},\average{\parameter}+a\right]$
\end{subnumcases}
This auxiliary function $h$ forms an inequality with the map $\costfunction$
	\begin{equation}
		-\cos(2\tau)h(\avgest{\parameter} + a\sin\tau) \leq -\cos(2\tau)\costfunction(\avgest{\parameter} + a\sin\tau)
	\end{equation} 
which meets with equality only when $\cos2\tau = 0$. The inequality is justified by the sign of $-\cos2\tau$ and the strict inequalities of \eqref{eq:strict-convexity} and Lemma~\ref{lem:convexity-first-order-condition}. The auxiliary function $h$ is show in Fig.~\ref{fig:h-auxiliary-function} for reference. A lower bound for \eqref{eq:average-hessian-estimate-alt} is provided in \eqref{eq:average-hessian-lower-bound} by using $h$ rather than $\costfunction$. This lower bound is a strictly positive function. The first and second terms are strictly positive by \eqref{eq:strict-convexity} and Prop.~\ref{prop:strictly-positive-hessian-integral}, respectively.
\begin{figure*}[!t]
\normalsize
\begin{eqnarray}
	\avgest{\hessian}(\avgest{\parameter}) &>& \frac{8}{a^2\pi}\int_{-\pi/2}^{-\pi/2} \left(-\cos2\tau\right)h(\avgest{\parameter} + a\sin\tau)d\tau \\
    \label{eq:average-hessian-lower-bound}
	&=& \frac{3(2\sqrt{2} - 1)}{4\pi a^2}\left(\frac{1}{2}\costfunction\left(\average{\parameter} + \frac{a}{\sqrt{2}}\right) + \frac{1}{2}\costfunction\left(\average{\parameter} - \frac{a}{\sqrt{2}}\right) - \costfunction\left(\average{\parameter}\right)\right) + \frac{3\sqrt{2}}{4\pi a} \int_{\avgest{\parameter}-a/\sqrt{2}}^{\avgest{\parameter}+a/\sqrt{2}} \costfunction''(\phi)d\phi > 0
\end{eqnarray}
\hrulefill
\end{figure*}
\begin{figure}[htb]
	\centering
    \resizebox{3.25in}{!}{\begin{tikzpicture}
	\draw[ultra thick, -latex] (-4.5, -1) --++(9, 0) node[anchor=west] {$\average{\parameter}$};
	\draw[ultra thick, -latex] (-5, -1) --++(0, 4.5) node[anchor=east] {$\costfunction(\average{\parameter})$};
	
	\draw[thick, black!60, dashed] ({-4}, -1.75) node[anchor=north, black] {$\average{\parameter} - a$} --++(0, {1.75 + 1.5/16*(-4 + 2)^2})  -- ++(0, 1);
	\draw[thick, black!60, dashed] ({-4/sqrt(2)}, -1.25) node[anchor=north, black] {$\average{\parameter} - \frac{a}{\sqrt{2}}$} --++(0, {1.25 + 1.5/16*(-4/sqrt(2) + 2)^2});
	\draw[thick, black!60, dashed] ({4}, -1.75) node[anchor=north, black] {$\average{\parameter} + a$} --++(0, {1.75 + 1.5/16*(4 + 2)^2})  -- ++(0, 0.1);
	\draw[thick, black!60, dashed] ({4/sqrt(2)}, -1.25) node[anchor=north, black] {$\average{\parameter} + \frac{a}{\sqrt{2}}$} --++(0, {1.25 + 1.5/16*(4/sqrt(2) + 2)^2});
	\draw[thick, black!60, dashed] ({0}, -1.75) node[anchor=north, black] {$\average{\parameter}$} --++(0, {1.75 + 1.5/16*(0 + 2)^2});
	
	\draw[ultra thick, domain=-4.2:4.1] plot (\x, {1.5/16*(\x + 2)^2)});
	
	\draw[ultra thick, dotted, red, domain=2.8:4] plot (\x, {(\x - 4/sqrt(2))*(1.5/8*(4/sqrt(2) + 2)) + (1.5/16*(4/sqrt(2) + 2)^2});
	\draw[ultra thick, dotted, red, domain=-4:-2.8] plot (\x, {(\x + 4/sqrt(2))*(1.5/8*(-4/sqrt(2) + 2)) + (1.5/16*(-4/sqrt(2) + 2)^2});
	
	\draw[ultra thick, dashed, blue] (0, {1.5/16*(2^2)}) -- ({4/sqrt(2)}, {1.5/16*(4/sqrt(2) + 2)^2});
	\draw[ultra thick, dashed, blue] (0, {1.5/16*(2^2)}) -- ({-4/sqrt(2)}, {1.5/16*(-4/sqrt(2) + 2)^2});
	
	\filldraw[black] ({-4/sqrt(2)}, {1.5/16*(-4/sqrt(2) + 2)^2}) circle (3pt);
	\filldraw[black] ({0}, {1.5/16*(2)^2}) circle (3pt);
	\filldraw[black] ({4/sqrt(2)}, {1.5/16*(4/sqrt(2) + 2)^2}) circle (3pt);
	
	\draw node[black] at (-1.5, 2.5) {
		\begin{tabular}{cl}
			{\tikz\draw[black, ultra thick] (0,0) -- (1, 0);}  & Graph of $\costfunction$ \\
			{\tikz\draw[red, dashed, ultra thick] (0,0) -- (1, 0);}  & Support Lines \\
			{\tikz\draw[blue, dashed, ultra thick] (0,0) -- (1, 0);}  & Convex Line Segments \\
			{\tikz\filldraw[black] (0,0) circle (3pt);} & Point on $(\average{\parameter}, \costfunction(\average{\parameter}))$ curve
		\end{tabular}
	};
\end{tikzpicture}} %
	\caption{The auxiliary function $h$ constructed from support lines and convex line segments for an example $\costfunction$.}
	\label{fig:h-auxiliary-function}
\end{figure}
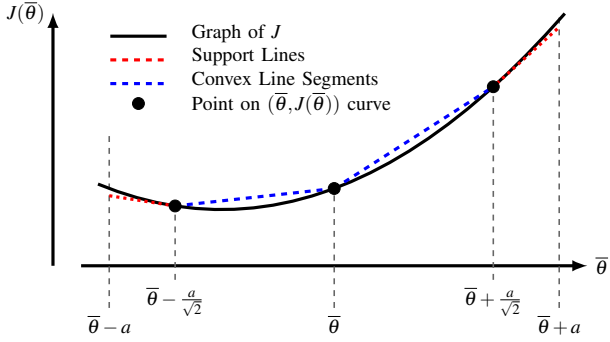
\end{proof}

\subsection{NESC Practical Stability}
	\label{sec:nesc-practical-stability}

With the existence and uniqueness of a stable equilibrium of the average NESC system established, we now give the stability results of this work.

\begin{proof}[Proof of Theorem~\ref{thm:nesc-sgpuas}]

First we change of variables from $(\estimate{\parameter},\estimate{\invhessian})$ into the error coordinates $(\error{\parameter},\error{\loginvhessian})$ where $\error{\parameter} = \estimate{\parameter} - \stationary{\parameter}$ and $\error{\loginvhessian} = \ln(\estimate{\invhessian}\cdot\average{\hessian}(\estimate{\parameter}))$. This transform is valid everywhere since $\estimate{\invhessian} > 0$ by the original domain and $\avgest{\hessian} > 0$ from Prop.~\ref{prop:average-hessian-strictly-positive}. The dynamics in the new variables are
\begin{eqnarray}
	\dot{\error{\parameter}} &=& -\gain e^{\error{\loginvhessian}}\cdot\frac{\estimate{\gradient}}{\avgest{\hessian}(\error{\parameter})}\\
	\dot{\error{\loginvhessian}} &=& \omega_l\left(1 - e^{\error{\loginvhessian}}\frac{\estimate{\hessian}}{\avgest{\hessian}(\error{\parameter})}\right) - \gain e^{\error{\loginvhessian}}\cdot\estimate{\gradient}\cdot\frac{\avgest{\hessian}'(\error{\parameter})}{\avgest{\hessian}^2(\error{\parameter})}.
\end{eqnarray}
The parameter estimate differential equation remains unchanged except that it now appears like a Newton Flow scaled by $e^{\error{\loginvhessian}}$. The differential equation for transformed  $\error{\loginvhessian}$ gains a second term which accounts for $\avgest{\hessian}$ varying as $\error{\parameter}$ changes. To prove the sGPUAS of this transformed coordinate system, one first establishes the GUAS of the average system. The average dynamics of the transformed coordinates are 
\begin{eqnarray}
	\label{eq:average-parameter-error-ode}
	\dot{\avgerror{\parameter}} &=& -\gain e^{\avgerror{\loginvhessian}}\cdot\frac{\avgest{\gradient}(\avgerror{\parameter})}{\avgest{\hessian}(\avgerror{\parameter})} \\
	\label{eq:average-log-invhessian-error-ode}
	\dot{\avgerror{\loginvhessian}} &=& \omega_l\left(1 - e^{\avgerror{\loginvhessian}}\right) - \gain e^{\avgerror{\loginvhessian}}\cdot\avgest{\gradient}(\avgerror{\parameter})\cdot\frac{\avgest{\hessian}'(\avgerror{\parameter})}{\avgest{\hessian}^2(\avgerror{\parameter})}
\end{eqnarray}

The candidate Lyapunov function in this average error system $(\avgerror{\parameter},\avgerror{\loginvhessian})$ is
\begin{equation}
	\label{eq:lyapunov-function}
	\lyapunov(\avgerror{\parameter},\avgerror{\loginvhessian}) =  \frac{1}{2}\avgerror{\parameter}^2 + \beta\int_0^{\avgerror{\parameter}} \sgn(\phi) \frac{\vert\average{\hessian}'(\phi)\vert}{\average{\hessian}(\phi)}d\phi + \ln\left(e^{\avgerror{\loginvhessian}} - \avgerror{\loginvhessian}\right)
\end{equation}
where $\sgn$ is the sign function and $\beta$ is chosen to satisfy
\begin{equation}
	\label{eq:beta-inequality}
	\max_{\avgerror{\loginvhessian}\in\realnumbers}\left\vert \frac{e^{\avgerror{\loginvhessian}} - 1}{e^{\avgerror{\loginvhessian}} - \avgerror{\loginvhessian}} \right\vert \leq \beta.
\end{equation} The first and third terms are radially unbounded, positive definite functions with respect to their arguments of $\avgerror{\parameter}$ and $\avgerror{\loginvhessian}$, respectively while the second term is positive semidefinite function with respect to its argument $\avgerror{\parameter}$. Hence, $V$ is a radially unbounded function. Taking the Lie derivative of \eqref{eq:lyapunov-function} results in
\begin{figure*}[!t]
\normalsize
\begin{eqnarray}
    \label{eq:lyanpuov-lie-derivative}
	\frac{d}{dt}\lyapunov(\avgerror{\parameter},\avgerror{\loginvhessian}) &=& - \gain e^{\avgerror{\loginvhessian}}\cdot\frac{\vert\avgest{\gradient}(\avgerror{\parameter})\cdot\avgest{\hessian}'(\avgerror{\parameter})\vert}{\avgest{\hessian}^2(\avgerror{\parameter})}\left(\beta - \frac{e^{\avgerror{\loginvhessian}} - 1}{e^{\avgerror{\loginvhessian}} - \avgerror{\loginvhessian}}\cdot\sgn(\avgerror{\parameter}\cdot\avgest{\hessian}'(\avgerror{\parameter}))\right) -\gain e^{\avgerror{\loginvhessian}}\cdot\frac{\avgerror{\gradient}(\avgerror{\parameter})\cdot\avgerror{\parameter}}{\avgerror{\hessian}(\avgerror{\parameter})} - \omega_l \frac{\left(e^{\avgerror{\loginvhessian}} - 1\right)^2}{e^{\avgerror{\loginvhessian}} - \avgerror{\loginvhessian}}
\end{eqnarray}
\hrulefill
\end{figure*}
\eqref{eq:lyanpuov-lie-derivative} where $\avgest{\gradient}(\avgerror{\parameter})\cdot\sgn(\avgerror{\parameter}) = \vert\avgest{\gradient}(\avgerror{\parameter})\vert$ by Prop.~\ref{prop:average-gradient-estimate-strictly-monotonically-increasing}. This is a continuous function as discontinuities from the sign function only occur when multiplied by $\avgest{\gradient}(\avgerror{\parameter})\cdot\avgest{\hessian}'(\avgerror{\parameter}) = 0$. The second and third term summation in \eqref{eq:lyanpuov-lie-derivative} give negative definite functions over the $(\avgerror{\parameter},\avgerror{\loginvhessian})$ plane. The first term is a negative semi-definite function given $\beta$ satisfies the functional bounds of \eqref{eq:beta-inequality}. Thus \eqref{eq:lyanpuov-lie-derivative} is a negative definite function. The Lyapunov function being radially unbounded positive definite function with a negative definite implies that the average system $(\avgerror{\parameter},\avgerror{\loginvhessian})$ is GUAS to the origin \cite[Th.~4.9]{ref:khalil-2002} and the original system $(\error{\parameter},\error{\loginvhessian})$ is consequently sGPUAS to the origin \cite{ref:teel-1998}.

For the practical exponential stability of the NESC, we consider a new set of error variables $(\error{\parameter}, \error{\invhessian})$ where $\error{\invhessian} = \estimate{\invhessian} - \stationary{\invhessian}$ and $\stationary{\invhessian} = 1/\average{\hessian}(\stationary{\parameter})$. The system dynamics are
\begin{eqnarray}
    \dot{\error{\parameter}} &=& -\gain \left(\error{\invhessian} + \stationary{\invhessian}\right) \estimate{\gradient} \\
    \dot{\error{\invhessian}} &=& -\omega_l \left(\error{\invhessian} + \stationary{\invhessian})\right)\left(1 - \left(\error{\invhessian} + \stationary{\invhessian}\right)\estimate{\hessian}\right)
\end{eqnarray}
and the average system dynamics are
\begin{eqnarray}
    \dot{\avgerror{\parameter}} &=& -\gain \left(\avgerror{\invhessian} + \stationary{\invhessian}\right) \estimate{\gradient} \\
    \dot{\avgerror{\invhessian}} &=& -\omega_l \left(\avgerror{\invhessian} + \stationary{\invhessian}\right)\left(1 - \left(\avgerror{\invhessian} + \stationary{\invhessian}\right)\average{\hessian}(\avgerror{\parameter})\right)
\end{eqnarray}
The linearization around the origin of this error system is
\begin{equation}
    \label{eq:linear-error-linearization}
    \frac{d}{dt}\begin{bmatrix} \avgerror{\parameter} \\ \avgerror{\invhessian}\end{bmatrix} = \begin{bmatrix}
        -\gain \frac{\average{\gradient}'(\stationary{\parameter})}{\average{\hessian}(\stationary{\parameter})} & 0 \\
        0 & -\omega_l
    \end{bmatrix}\begin{bmatrix}\avgerror{\parameter} \\ \avgerror{\invhessian}\end{bmatrix}
\end{equation}
Trivially the eigenvalues of the coefficient matrix of \eqref{eq:linear-error-linearization} has eigenvalues of $-\gain \average{\gradient}'(\stationary{\parameter})/\average{\hessian}(\stationary{\parameter})$ and $-\omega_l$. All of these eigenvalues are negative since $\average{\gradient}' > 0$ and $\average{\hessian} > 0$ by Prop.~\ref{prop:average-gradient-estimate-strictly-monotonically-increasing} and \ref{prop:average-hessian-strictly-positive} respectively, so the coefficent matrix is Hurwitz. The origin is then exponentially stable for the average system \cite[Thm 4.7]{ref:khalil-2002}. Coupled with \cite[Thm 10.4]{ref:khalil-2002}, the $(\error{\parameter},\error{\invhessian})$ is exponentially stable to a periodic orbit in a neighborhood of radius $\mathcal{O}(\omega^{-1})$ which satisfies the definition of practical exponential stability provided in Appendix \ref{sec:stability-definitions}.
\end{proof}

\section{Conclusions}

We have demonstrated the practical stability of the NESC, both the sGPUAS and the local practical exponential stability to unique equilibrium point for scalar maps which are strictly convex, even when for those not strongly convex. Critically, the average of the perturbation-based gradient and Hessian estimates retain important properties that allow for model-free ESC to achieve stronger practical stability on these maps than their model-based analogs would suggest. Specifically, the NESC algorithm is sGPUAS and PES to the equilibrium of the system described by \eqref{eq:parameter-estimate-ode} and \eqref{eq:invhessian-estimate-ode}.

\bibliographystyle{IEEEtranS}
\bibliography{refs}

\appendices
\section{Practical Stability}
	\label{sec:stability-definitions}
	
The stability of systems whose state dynamics depend on a vector of small perturbation parameter are often studied under practical stability. In this work, we consider systems  of the singular small parameter $\omega^{-1}$, i.e.\ the arbitrarily fast dither rates $\omega$. These parameter dependent systems satisfy
\begin{equation}
	\label{eq:fast-oscillation-ode}
	\frac{d}{dt} x = f^{\omega}(t, x)
\end{equation}
where $x(t)$ is a solution to \eqref{eq:fast-oscillation-ode}. The solutions of the parameter dependent \eqref{eq:fast-oscillation-ode} may converge to a close neighborhood of the origin as $t\to\infty$ given the initial conditions of $x$ are sufficiently close to the origin. The notion of practical stability of the origin is that one can find an interval of $\omega$ such all trajectories starting in an arbitrarily large distance from the origin converge towards arbitrarily small neighborhood of the origin. The formal definitions for sGPUAS and PES are given below.

\begin{defn}[{\cite{ref:labar-2019}}]
    The origin of \eqref{eq:fast-oscillation-ode} with parameter $\omega$ is semiglobally practically uniformly asymptotically stable (sGPUAS) if for every bounded neighborhood of the origin $\set{B}\subset\realnumbers^n$ and $\set{V}\subset\realnumbers^n$, there exist bounded neighborhoods of the origin $\set{Q}\subset\set{B}$ and $\set{W}$ such that $\set{V}\subset\set{W}\subset\realnumbers^n$ and a $\limitvalue{\omega}$ such that $\forall\omega\in(\limitvalue{\omega},\infty)$, there exists $T$ such that $\forall t_0\in\realnumbers_{>0}$ the following hold:
    \begin{enumerate}
        \item (Boundedness) $x(t_0)\in\set{V}\implies x(t)\in\set{W},\ \forall t \geq t_0$;
        \item (Stability) $x(t_0)\in\set{Q}\implies x(t)\in\set{B},\ \forall t \geq t_0$;
        \item (Practical Convergence) $x(t_0)\in\set{V}\implies x(t)\in\set{B}$ $\forall~t~\geq~t_0~+~T$
    \end{enumerate}
\end{defn}

\begin{defn}
    \label{def:deua}
    The origin of \eqref{eq:fast-oscillation-ode} is said to be practically exponentially stable (PES) for $\delta$ sufficiently small if there exists $M,\lambda>0$ such that for every $\varepsilon > 0$ there exists $\limitvalue{\omega}$ such that for all $\omega\in (\limitvalue{\omega},\infty)$ and $t\in[t_0,\infty)$
    \begin{equation}
        \label{eq:dlpes}
        \Vert x(t_0) \Vert < \delta \implies \Vert x(t) \Vert \leq M \Vert x(t_0) \Vert e^{-\lambda(t-t_0)} + \varepsilon
    \end{equation}
\end{defn}

Without dependence on the small parameter vector, the practical notion is removed and stability follows more conventional and stronger stability definitions found in \cite[Ch~4]{ref:khalil-2002}.

\section{Convexity Lemmas}
	\label{sec:convexity-lemmas}
	
The following lemmas and proposition are useful for establishing the first order conditions and second order conditions of convex and strictly convex functions.

\begin{lem}[\cite{ref:boyd-2004}]
	\label{lem:convexity-first-order-condition}
	A differentiable map is strictly convex if and only if for $\parameter_1,\ \parameter_2\in\realnumbers$ where $\parameter_1\neq\parameter_2$ the following inequality holds.
	\begin{equation}
		\costfunction(\parameter_2) > \costfunction(\parameter_1) + \costfunction'(\parameter_1)\cdot(\parameter_2 - \parameter_1)
	\end{equation}

\end{lem}

\begin{lem}[\cite{ref:boyd-2004}]
	\label{lem:convexity-second-order-condition}
	A twice differentiable map is convex if and only if 
	\begin{equation}
		\costfunction'' \geq 0,\quad\forall\ \parameter\in\realnumbers
	\end{equation}
\end{lem}

\begin{prop}
	\label{prop:strictly-positive-hessian-integral}
	A scalar, strictly convex map $\costfunction$ implies that $\forall\ \parameter_1,\parameter_2\in\realnumbers$ where $\parameter_1 < \parameter_2$, the following proper integral of the second derivative of the map is strictly positive.
	\begin{equation}
		\label{eq:proposition-strictly-positive-hessian-integral}
		\int_{\parameter_1}^{\parameter_2}\costfunction''(\phi)d\phi > 0
	\end{equation}
\end{prop}
\begin{proof}
	Suppose the proposition was not true, then $\costfunction$ is strictly convex and there is a finite interval where $\int_{\parameter_1}^{\parameter_2} \costfunction''(\phi) d\phi \leq 0$. From Lemma~\ref{lem:convexity-second-order-condition}, the integrand must be strictly non-negative so $\costfunction''(\parameter) = 0$ for all $\parameter\in(\parameter_1,\parameter_2)$, i.e.\ the function is a linear function on the interval. This means that $\costfunction$ cannot be strictly convex as \eqref{asmp:strictly-convex} would be an equality rather than the strict inequality, leading to the contradiction that $\costfunction$ was strictly convex. Therefore, \eqref{eq:proposition-strictly-positive-hessian-integral} must hold for any $\parameter_1,\parameter_2\in\realnumbers$ where $\parameter_1 < \parameter_2$.
\end{proof}

\end{document}